\documentclass[a4paper,12pt]{amsproc}
\usepackage{amssymb}
\usepackage{graphicx}
\usepackage[cmtip,all]{xy}

\newtheorem{theorem}{Theorem}[section]
\newtheorem{lemma}[theorem]{Lemma}
\newtheorem{prop}[theorem]{Proposition}
\newtheorem{cor}[theorem]{Corollary}
\newtheorem{con}[theorem]{Conjecture}

\theoremstyle{definition}

\newtheorem{prob}[theorem]{Problem}

\theoremstyle{remark}
\newtheorem{remark}[theorem]{Remark}

\numberwithin{equation}{section}

\begin{document}

\title[Classification of special Anosov endomorphisms]{Classification of special Anosov endomorphisms of nil-manifolds}

\author[S.M.Moosavi]{Seyed Mohsen Moosavi}
\address{Department of Mathematics, Faculty of Mathematical Sciences, Tarbiat Modares University,
Tehran 14115-134, Iran}
\email{seyedmohsen.moosavi@modares.ac.ir, smohsenmoosavi2009@gmail.com}
\author[K.Tajbakhsh]{Khosro Tajbakhsh}
\address{Department of Mathematics, Faculty of Mathematical Sciences, Tarbiat Modares University,
Tehran 14115-134, Iran}
\email{khtajbakhsh@modares.ac.ir, arash@cnu.ac.kr}

\date{}

\begin{abstract}In this paper we give a classification of special endomorphisms of nil-manifolds: Let $f:N/\Gamma \rightarrow N/\Gamma$ be a covering map of a nil-manifold and denote by $A:N/\Gamma \rightarrow N/\Gamma$ the nil-endomorphism which is homotopic to $f$. If $f$ is a special $TA$-map, then $A$ is a hyperbolic nil-endomorphism and $f$ is topologically conjugate to $A$.
\end{abstract}

\maketitle

\section{Introduction}
Finding a universal model for Anosov diffeomorphisms has been an important problem in dynamical systems. In this general context, Franks and Manning proved that every Anosov diffeomorphism of an infra-nil-manifold is topologically conjugate to a hyperbolic infra-nil-automorphism \cite{[Fr 1],[Fr 2],[Ma 1],[Ma 2]} (According to Dekimpe's work \cite{[KD2]}, some of their results are incorrect). Based on this result, Aoki and Hiraide has been studied the dynamics of covering maps of a torus \cite{[Ao-Hi]}. The importance of infra-nil-manifolds comes from the following  Conjecture \ref{con} and Theorem \ref{thm} :

The first non-toral example of an Anosov diffeomorphism was constructed by S. Smale in \cite{[Sm]}. He conjectured that, up to topologically conjugacy, the construction in Smale's example gives every possible Anosov diffeomorphism on a closed manifold.

\begin{con}\label{con}
Every Anosov diffeomorphism of a closed manifold is topologically conjugate to a hyperbolic affine infra-nil-automorphism.
\end{con}

\begin{theorem}[Gromov \cite{[Gr]}]\label{thm}
Every expanding map on a closed manifold is topologically conjugate to an expanding affine infra-nil-endomorphism.
\end{theorem}
The conjecture has been open for many years (see \cite{[ِDe]} page 48). An interesting problem is to consider the conjecture for endomorphisms of a closed manifold. Our main theorem is a partial answer to the conjecture.\\

In this paper we give a classification of special endomorphisms of nil-manifolds. Infact, Aoki and Hiraide \cite{[Ao-Hi]} in 1994 proposed two problems:
\begin{prob} \label{problem1}
Is every special Anosov differentiable map of a torus topologically conjugate to a hyperbolic toral endomorphism?
\end{prob}
\begin{prob} \label{problem2}
Is every special topological Anosov covering map of an arbitrary closed topological manifold topologically conjugate to a hyperbolic infra-nil-
endomorphism of an infra-nil-manifold ?
\end{prob}

Aoki and Hiraide answered problem \ref{problem1} partially as follows:
\begin{theorem}[\cite{[Ao-Hi]} Theorem 6.8.1] \label{Aoki Theorem 6.8.1}
Let $f: \mathbb{T}^n \rightarrow \mathbb{T}^n$ be a $TA$-covering map of an $n$-torus and denote by $A: \mathbb{T}^n \rightarrow \mathbb{T}^n$ the toral endomorphism homotopic to $f$. Then $A$ is hyperbolic. Furthermore the inverse limit system of $(\mathbb{T}^n,f)$ is topologically conjugate to the inverse limit system of $(\mathbb{T}^n,A)$.
\end{theorem} \label{[Ao-Hi] Theorem 6.8.1}

\begin{theorem}[\cite{[Ao-Hi]} Theorem 6.8.2] \label{Aoki Theorem 6.8.2}
Let $f$ and $A$ be as Theorem \ref{Aoki Theorem 6.8.1}. Suppose $f$ is special, then the following statements hold:
\begin{enumerate}
\item[(1)] if $f$ is a $TA$-homeomorphism, then $A$ is a hyperbolic toral automorphism and $f$ is topologically conjugate to $A$,
\item[(2)] if $f$ is a topological expanding map, then $A$ is an expanding toral endomorphism and $f$ is topologically conjugate to $A$,
\item[(3)] if $f$ is a strongly special $TA$-map, then $A$ is a hyperbolic toral endomorphism and $f$ is topologically conjugate to $A$.
\end{enumerate}
\end{theorem}

In \cite{[Su 1]}, Sumi has altered the condition " strongly special "  (part (3) of Theorem \ref{Aoki Theorem 6.8.2}) to just " special "  as follows:

\begin{theorem}[\cite{[Su 1]}] \label{Su Theorem}
Let $f$ and $A$ be as Theorem \ref{Aoki Theorem 6.8.1}.
If $f$ is a special $TA$-map, then $A$ is a hyperbolic toral endomorphism and $f$ is topologically conjugate to $A$.
\end{theorem}

In \cite{[Su 2]}, Sumi generalized (\textit{\underline{incorrectly}}) Theorem \ref{Aoki Theorem 6.8.1} and parts (1) and (2) of Theorem \ref{Aoki Theorem 6.8.2} for infra-nil-manifolds as follows:

\begin{theorem}[\cite{[Su 2]} Theorem 1] \label{Su2 theorem1}
Let $f:N/\Gamma \rightarrow N/\Gamma$ be a covering map of an infra-nil-manifold and denote as $A:N/\Gamma \rightarrow N/\Gamma$ the infra-nil-endomorphism homotopic to $f$. If $f$ is a TA-map, then $A$ is hyperbolic and the inverse limit system of $(N/\Gamma, f)$ is topologically conjugate to the inverse limit system of $(N/\Gamma, A)$ .
\end{theorem}

\begin{theorem}[\cite{[Su 2]} Theorem 2] \label{Su2 theorem2}
Let $f$ and $A$ be as in Theorem \ref{Su2 theorem1}. Then the following statements hold:
\begin{enumerate}
\item[(1)]if $f$ is a TA-homeomorphism, then $A$ is a hyperbolic infra-nil-automorphism and $f$ is topologically conjugate to $A$,
\item[(2)]if $f$ is a topological expanding map, then $A$ is an expanding infra-nil-endomorphism and $f$ is topologically conjugate to $A$.
\end{enumerate}
\end{theorem}

Dekimpe \cite{[KD2]}, expressed that there might exist (interesting) diffeomorphisms and self-covering maps of an infra-nil-manifold which are not even homotopic to an infra-nil-endomorphism. Dekimpe \cite{[KD2]} in \S 4, gave an expanding map not topologically conjugate to an infra-nil-endomorphism. And in \S 5, he gave an Anosov diffeomorphism not topologically conjugate to an infra-nil-automorphism. According to \cite{[KD2]}, Theorem \ref{Su2 theorem1} and Theorem \ref{Su2 theorem2} are true for nil-manifolds. Of course, if in Sumi's works, the map $f$ has a desired homotopic infra-nil-endomorphism, then the theorems hold.\\

Since, nil-manifolds are included in infra-nil-manifolds, we consider \cite{[Su 2]} for nil-manifolds.\\
 
In the paper, by using Theorem \ref{Su Theorem}, we partially answer problem \ref{problem2} of Aoki and Hiraide as follows:

\begin{theorem}[Main Theorem] \label{Theorem Main}
Let $f:N/\Gamma \rightarrow N/\Gamma$ be a covering map of a nil-manifold and denote as $A:N/\Gamma \rightarrow N/\Gamma$ the nil-endomorphism homotopic to $f$ (according to \cite{[KD2]}, such a unique homotopy exists for nil-manifolds). If $f$ is a special $TA$-map, then $A$ is a hyperbolic nil-endomorphism and $f$ is topologically conjugate to $A$.
\end{theorem}

\begin{cor}\label{Mcor}
If $f:N/\Gamma \rightarrow N/\Gamma$ is a special Anosov endomorphism of a nil-manifold then it is conjugate to a hyperbolic nil-endomorphism.
\end{cor}

\section{Preliminaries}
Let $X$ and $Y$ be compact metric spaces and let $f:X \rightarrow X$ and $g:Y \rightarrow Y$ be continuous surjections. Then $f$ is said to be topologically conjugate to $g$ if there exists a homeomorphism $\varphi:Y \rightarrow X$ such that $f\circ\varphi=\varphi\circ g$.\\

Let $X$ be a compact metric space with metric $d$ . For $f:X \rightarrow X$ a continuous surjection, we let
\begin{align*}
&X_{f}=\{\tilde{x}=(x_{i}):x_{i}\in X \text{ and } f(x_{i})=x_{i+1}, i\in \mathbb{Z} \},\\
&\sigma_{f}((x_{i}))=(f(x_{i})).
\end{align*}
The map $\sigma_{f} : X_{f}\rightarrow X_{f}$ is called the \emph{shift map} determined by $f$. We call $(X_{f}, \sigma_{f})$ the \emph{inverse limit} of $(X, f)$. A homeomorphism $f:X\rightarrow X$ is called \emph{expansive} if there is a constant $e>0$ (called an \emph{expansive constant}) such that if $x$ and $y$ are any two distinct points of $X$ then $d(f^i(x), f^i(y)) > e$ for some integer $i$. A continuous surjection $f:X\rightarrow X$ is called \emph{$c$-expansive} if there is a constant $e>0$ such that for $\tilde{x},\tilde{y}\in X_{f}$ if $d(x_{i}, y_{i})\leq e$ for all $i\in \mathbb{Z}$ then $\tilde{x}=\tilde{y}$. In particular, if there is a constant $e>0$ such that for $x,y\in X$ if $d(f^{i}(x), f^{i}(y))\leq e$ for all $i\in \mathbb{N}$ then $x=y$, we say that $f$ is \emph{positively expansive}. A sequence of points $\{x_{i} : a<i<b\}$ of $X$ is called a $\delta$-pseudo orbit of $f$ if $d(f(x_{i}), x_{i+1})<\delta$ for $i\in(a, b-1)$. Given $\epsilon>0$ a $\delta$-pseudo orbit of $\{x_{i}\}$ is called to be $\epsilon$-traced by a point $x\in X$ if $d(f^{i}(x), x_{i})<\epsilon$ for every $i\in(a, b-1)$ . Here the symbols $a$ and $b$ are taken as $-\infty\leq a<b\leq\infty$ if $f$ is bijective and as $-1\leq a<b\leq\infty$ if $f$ is not bijective. $f$ has the \emph{pseudo orbit tracing property} (abbrev. POTP) if for every $\epsilon>0$ there is $\delta>0$ such that every $\delta$-pseudo orbit of $f$ can be $\epsilon$-traced by some point of $X$.

We say that a homeomorphism $f:X\rightarrow X$ is a \emph{topological Anosov map} (abbrev. $TA$-map) if $f$ is expansive and has POTP. Analogously, We say that a continuous surjection $f:X\rightarrow X$ is a \emph{topological Anosov map} if $f$ is $c$-expansive and has POTP, and say that $f$ is a \emph{topological expanding map} if $f$ is positively expansive and open. We can check that every topological expanding map is a $TA$-map (see \cite{[Ao-Hi]} Remark 2.3.10).\\

Let $X$ and $Y$ be metric spaces. A continuous surjection $f:X\rightarrow Y$ is called a \emph{covering map} if for $y\in Y$ there exists an open neighborhood $V_y$ of $y$ in $Y$ such that
\begin{equation*}
f^{-1}(V_y)=\bigcup_i U_i\quad (i\neq i'\Rightarrow U_i\cap U_i'=\emptyset)
\end{equation*}
where each of $U_i$ is open in $X$ and $f_{|U_i} : U_i\rightarrow V_y$ is a homeomorphism. A covering map $f:X\rightarrow Y$ is especially called a \emph{self-covering map} if $X = Y$. We say that a continuous surjection $f:X\rightarrow Y$ is a local homeomorphism if for $x\in X$ there is an open neighborhood $U_x$, of $x$ in $X$ such that $f(U_x)$ is open in $Y$ and $f_{|U_x} : U_x\rightarrow f(U_x)$ is a homeomorphism. It is clear that every covering map is a local homeomorphism. Conversely, if $X$ is compact, then a local homeomorphism $f:X\rightarrow Y$ is a covering map (see \cite{[Ao-Hi]} Theorem 2.1.1).\\

Let $\pi:Y\rightarrow X$ be a covering map. A homeomorphism $\alpha:Y\rightarrow Y$ is called a \emph{covering transformation} for $\pi$ if $\pi\circ  \alpha =\pi$ holds. We denote as $G(\pi)$ the set of all covering transformations for $\pi$. It is easy to see that $G(\pi)$ is a group, which is called the \emph{covering transformation group} for $\pi$.\\

Let $M$ be a closed smooth manifold and let $C^{1}(M, M)$ be the set of all $C^{1}$ maps of $M$ endowed with the $C^1$ topology. A map $f\in C^1(M, M)$ is called an \emph{Anosov endomorphism} if $f$ is a $C^{1}$ regular map and if there exist $C>0$ and
$0<\lambda<1$ such that for every $\tilde{x}=(x_i)\in M_{f}=\{\tilde{x}=(x_i):x_i\in M \text{ and } f(x_i)=x_{i+1},\; i\in \mathbb{Z}\}$ there is a splitting
\begin{equation*}
T_{x_i}M=E_{x_i}^{s}\oplus E_{x_i}^{u},\quad i\in \mathbb{Z}
\end{equation*}
(we show this by $T_{\tilde{x}}M=\bigcup_i(E_{x_i}^{s}\oplus E_{x_i}^{u})$) so that for all $i\in \mathbb{Z}$:
\begin{enumerate}
\item[(1)]$D_{x_i}f(E_{x_i}^\sigma)=E_{x_{i+1}}^\sigma$ where $\sigma=s,u$,
\item[(2)]for all $n\geq 0$
	\begin{align*}
	&\parallel D_{x_i}f^n(v)\parallel \leq C\lambda^n\parallel v\parallel\text{ if }v\in E_{x_i}^s,\\
	&\parallel D_{x_i}f^n(v)\parallel \geq C^{-1}\lambda^{-n}\parallel v\parallel\text{ if }v\in E_{x_i}^u.
	\end{align*}
\end{enumerate}
If, in particular, $T_{\tilde{x}}M= \bigcup_i E_{x_i}^u$ for all $\tilde{x}=(x_{i})\in M_{f}$ , then $f$ is said to be \emph{expanding differentiable map}, and if an Anosov endomorphism $f$ is injective then $f$ is called an \emph{Anosov diffeomorphism}. We can check that every Anosov endomorphism is a TA-map, and that every expanding differentiable map is a topological expanding map (see \cite{[Ao-Hi]} Theorem 1.2.1).\\

A map $f\in C^1(M, M)$ is said to be \emph{$C^1$-structurally stable} if there is an open neighborhood $\mathcal{N}(f)$ of $f$ in $C^1(M, M)$ such that $g\in \mathcal{N}(f)$ implies that $f$ and $g$ are topologically conjugate. Anosov \cite{[An]} proved that every Anosov diffeomorphism is $C^1$-structurally stable, and Shub \cite{[Sh]} showed the same result for expanding differentiable maps. However, Anosov endomorphisms which are not diffeomorphisms nor expanding do not be $C^1$-structurally stable (\cite{[Ma-Pu]},\cite{ [Pr]}).\\ 

A map $f\in C^1(M, M)$ is said to be \emph{$C^1$-inverse limit stable} if there is an open
neighborhood $\mathcal{N}(f)$ of $f$ in $C^1(M, M)$ such that $g\in \mathcal{N}(f)$ implies that the inverse limit $(M_f , \sigma_f)$ of $(M, f)$ and the inverse limit $(M_g, \sigma_g)$ of $(M, g)$ are topologically conjugate. Man\'e and Pugh \cite{[Ma-Pu]} proved that every Anosov endomorphism is $C^1$-inverse limit stable.\\

We define \emph{special TA-maps} as follows. Let $f:X\rightarrow X$ be a continuous
surjection of a compact metric space. Define the \emph{stable} and \emph{unstable} sets
\begin{align*}
W^s(x)&= \{y\in X : \lim_{n \rightarrow \infty}{d(f^n(x)}, f^n(y))=0\},\\
W^u(\tilde{x})&= \{y_0\in X:\exists\tilde{y}=(y_i)\in X_f \text{ s.t. } \lim_{i \rightarrow\infty}d(x_{-i}, y_{-i})=0\}.
\end{align*}
for $x\in X$ and $\tilde{x}\in X_f$. A TA-map $f:X \rightarrow X$ is special if $f$ satisfies the
property that $W^u(\tilde{x})=W^u(\tilde{y})$ for every $\tilde{x},\tilde{y}\in X_f$ with $x_0=y_0$. Every hyperbolic nil-endomorphism is a special TA-covering map (See \cite{[Su 2]} Remark 3.13). By this and Theorem \ref{Theorem Main} We have the following corollary:
\begin{cor}
A TA-covering map of a nil-manifold is special if and only if it is conjugate to a hyperbolic nil-endomorphism.
\end{cor}

A \emph{Lie group} is a smooth manifold obeying the group properties and that satisfies the additional condition that the group operations are differentiable. Let $N$ be a Lie group. A vector field $X$ on $N$ is said to be \emph{invariant under left translations} if for each $g,h\in N, \;(dl_g)_h(X_h)=X_{gh}$, where $(dl_g)_h:T_hN\rightarrow T_{gh}N$ and $l_g:N\rightarrow N;\; x\mapsto gx$. A \emph{Lie algebra} $g$ is a vector space over some field $F$ together with a binary operation $[\cdot,\cdot]:g \times g \rightarrow g$ called the \emph{Lie bracket}, that satisfies:
\begin{enumerate}
\item[(1)] Bilinearity: $[ax+by,z]=a[x,z]+b[y,z],\; [z,ax+by]=a[z,x]+b[z,y] \; \forall x, y, z \in g$,
\item[(2)] Alternativity: $[x,x]=0 \; \forall x\in g$,
\item[(3)]The Jacobi Identity: $[x,[y,z]]+[z,[x,y]]+[y,[z,x]]=0 \; \forall x,y,z \in g$
\end{enumerate}
Let $Lie(N)$ be the set of all left-translation-invariant vector fields on $N$. It is a real vector space. Moreover, it is closed under Lie bracket. Thus $Lie(N)$ is a Lie subalgebra of the Lie algebra of all vector fields on $N$ and is called the \emph{Lie algebra} of $N$. A \emph{nilpotent} Lie group is a Lie group which is connected and whose Lie algebra is a nilpotent Lie algebra. That is, its Lie algebra's central series eventually vanishes.

A group $G$ is a \emph{torsion group} if every element in $G$ is of finite order. $G$ is called \emph{torsion free} if no element other than identity is of finite order. A \emph{discrete} subgroup of a topological group $G$ is a subgroup $H$ such that there is an open cover of $H$ in which every open subset contains exactly one element of $H$. In other words, the subspace topology of $H$ in $G$ is the discrete topology. A \emph{uniform} subgroup $H$ of $G$ is a closed subgroup such that the quotient space $G/H$ is compact.\\

We bring here the definitions of nil-manifolds and infra-nil-manifolds from Karel Dekimpe in \cite{[KD]} and \cite{[KD2]}.\\ 

Let $N$ be a Lie group and $Aut(N)$ be the set of all automorphisms of $N$. Assume
that $\overline{A}\in Aut(N)$ is an automorphism of $N$, such that there exists a discrete and cocompact subgroup
$\Gamma$ of $N$, with $\overline{A}(\Gamma)\subseteq \Gamma$. Then the space of left cosets $N/\Gamma$ is a closed manifold, and $\overline{A}$ induces an
endomorphism $A:N/\Gamma\rightarrow N/\Gamma,\; g\Gamma\mapsto \overline{A}(g)\Gamma$.
If we want this endomorphism to be Anosov, $\overline{A}$ must be hyperbolic (i.e. has no eigenvalue with modulus 1). It is known that this can happen only when $N$ is nilpotent. So we restrict ourselves to that case, where the resulting manifold $N/\Gamma$ is said to be a \emph{nil-manifold}. Such an endomorphism $A$ induced by an automorphism $\overline{A}$ is called a \emph{nil-endomorphism} and is said to be a \emph{hyperbolic nil-automorphism}, when $\overline{A}$ is hyperbolic. If in the above definition, $\overline{A}(\Gamma)= \Gamma$, the induced map is called a \emph{nil-automorphism}.

All tori, $\mathbb{T}^n=\mathbb{R}^n/\mathbb{Z}^n$ are examples of nil-manifolds.\\

Let $X$ be a topological space and let $G$ be a group. We say that $G$ \emph{acts} (continuously) on $X$ if to $(g, x)\in G \times X$ there corresponds a point $g\cdot x$ in $X$ and the following conditions are satisfied:
\begin{enumerate}
\item[(1)]$e\cdot x = x$ for $x\in X$ where $e$ is the identity,
\item[(2)]$g\cdot(g'\cdot x) = gg'\cdot x$  for $x\in X$ and $g,g'\in G$,
\item[(3)]for each $g\in G$ a map $x\mapsto g\cdot x$ is a homeomorphism of $X$.
\end{enumerate}
When $G$ acts on $X$, for $x,y\in X$ letting
\begin{equation*}
x\sim y \Leftrightarrow y=g\cdot x \text{ for some } g\in G
\end{equation*}
an equivalence relation $\sim$ in $X$ is defined. Then the identifying space $X/\sim$,
denoted as $X/G$, is called the orbit space by $G$ of $X$. It follows that for $x\in X$,
$[x]=\{g\cdot x : g\in G\}$ is the equivalence class.

An action of $G$ on $X$ is said to be \emph{properly discontinuous} if for each $x\in X$ there exists a neighborhood $U(x)$ of $x$ such that $U(x) \cap gU(x) = \emptyset$ for all $g\in G$ with $g \neq e_G$. Here $gU(x) =\{g\cdot y : y \in U(x)\}$.\\

Now we give an extended definition of nil-manifolds. Let $N$ be a connected and simply connected nilpotent Lie group and $Aut(N)$ be the group of continuous automorphisms of $N$. Then $Aff(N)= N\rtimes Aut(N)$ acts on $N$ in the following way:
\begin{equation*}
\forall(n,\gamma)\in Aff(N),\forall x\in N:(n,\gamma).x=n\gamma(x).
\end{equation*}

So an element of $Aff(N)$ consists of a translational part $n\in N$ and a linear part $\gamma\in Aut(N)$ (as a set $Aff(N)$ is just $N\times Aut(N)$) and $Aff(N)$ acts on $N$ by first applying the linear part and then multiplying on the left by the translational part). In this way, $Aff(N)$ can also be seen as a subgroup of $Diff(N)$.

Now, let $C$ be a compact subgroup of $Aut(N)$ and consider any torsion free discrete subgroup
$\Gamma$ of $N\rtimes C$, such that the orbit space $N/\Gamma$ is compact. Note that $\Gamma$ acts on $N$ as being also a subgroup of $Aff(N)$.
The action of $\Gamma$ on $N$ will be free and properly discontinuous, so $N/\Gamma$ is a manifold, which is called an \emph{infra-nil-manifold}.

Klein bottle is an example of infra-nil-manifolds.\\

In what follows, we will identify $N$ with the subgroup $N \times \{id\}$ of $N\rtimes Aut(N) = Aff(N)$, $F$ with the subgroup $\{id\} \times F$ and $Aut(N)$ with the subgroup $\{id\} \times Aut(N)$.

It follows from Theorem 1 of L. Auslander in \cite{[Au]}, that $\Gamma \cap N$ is a uniform lattice of $N$ and that $ \Gamma / (\Gamma \cap N)$ is a finite group. This shows that the fundamental group of an infra-nil-manifold $N/\Gamma$ is virtually nilpotent (i.e. has a nilpotent normal subgroup of finite index). In fact $\Gamma \cap N$ is a maximal nilpotent subgroup of $\Gamma$ and it is the only normal subgroup of $\Gamma$ with this property. (This also follows from \cite{[Au]}).

If we denote by $p : N\rtimes C\rightarrow C$ the natural projection on the second factor, then $p(\Gamma)=\Gamma \cap N$ is a uniform lattice of $N$ and that $ \Gamma / (\Gamma \cap N)$. Let $F$ denote this finite group $p(\Gamma)$, then we will refer to $F$ as being the \emph{holonomy group of $\Gamma$} (or of the infra-nil-manifold $N/\Gamma$). It follows that $\Gamma \subseteq N\rtimes F$. In case $F = \{id\}$, so $\Gamma \subseteq N$, the manifold $N/\Gamma$ is a nil-manifold. Hence, any infra-nil-manifold $N/\Gamma$ is finitely covered by a nil-manifold $N/(\Gamma \cap N)$. This also explains the prefix "infra".

Fix an infra-nil-manifold $N/\Gamma$, so $N$ is a connected and simply connected nilpotent
Lie group and $\Gamma$ is a torsion free, uniform discrete subgroup of $N\rtimes F$, where $F$ is a finite subgroup of $Aut(N)$. We will assume that $F$ is the holonomy group of $\Gamma$ (so for any $\mu\in F$, there exists an $n\in N$ such that $(n, \mu)\in \Gamma)$.

We can say that an element of $\Gamma$ is of the form $n\mu$ for some $n\in N$ and some $\mu\in F$. Also, any element of $Aff(N)$
can uniquely be written as a product $n\psi$, where $n\in N$ and $\psi\in Aut(N)$. The product in $Aff(N)$ is then given as
\begin{equation*}
\forall n_1,n_2\in N, \forall \psi_1,\psi_2\in Aut(N):n_1\psi_1n_2\psi_2=n_1\psi_1(n_2)\psi_1\psi_2.
\end{equation*}

Now we can define infra-nil-endomorphisms as follows:\\

Let $N$ be a connected, simply connected nilpotent Lie group and $F\subseteq Aut(N)$ a finite group. Assume that $\Gamma$ is a torsion free, discrete and uniform subgroup of $N\rtimes F$. Let $\overline{\mathcal{A}}:N\rtimes F\rightarrow N\rtimes F$ be an automorphism, such that $\overline{\mathcal{A}}(F) = F$ and $\overline{\mathcal{A}}(\Gamma)\subseteq \Gamma$, then, the map
\begin{equation*}
A:N/\Gamma\rightarrow N/\Gamma,\; \Gamma \cdot n \mapsto \Gamma \cdot \overline{\mathcal{A}}(n).
\end{equation*}
is the \emph{infra-nil-endomorphism} induced by $\overline{\mathcal{A}}$. In case $\overline{\mathcal{A}}(\Gamma)= \Gamma$, we call $A$ an \emph{infra-nil-automorphism}.\\

In the definition above, $\Gamma \cdot n$ denotes the orbit of n under the action of $\Gamma$. The computation above shows that $A$ is well defined. Note that infra-nil-automorphisms are diffeomorphisms, while in general an infra-nil-endomorphism is a self-covering map.

The following theorem shows that the only maps of an infra-nil-manifold, that lift to an automorphism of the corresponding nilpotent Lie group are exactly the infra-nil-endomorphisms defined above.

\begin{theorem}[\cite{[KD2]} Theorem 3.4]
Let $N$ be a connected and simply connected nilpotent Lie group, $F \subseteq Aut(N)$ a
finite group and $\Gamma$ a torsion free discrete and uniform subgroup of $N\rtimes F$ and assume that the holonomy group of $\Gamma$ is $F$. If $\overline{A}: N \rightarrow N$ is an automorphism for which the map 
\begin{equation*}
A:N/\Gamma\rightarrow N/\Gamma,\; \Gamma \cdot n \mapsto \Gamma \cdot \overline{A}(n).
\end{equation*}
is well defined (meaning that $\Gamma \cdot \overline{A}(n) = \Gamma \cdot \overline{A}(\gamma \cdot n)$ for all $\gamma\in \Gamma$), then
\begin{equation*}
\overline{\mathcal{A}}:N\rtimes F\rightarrow N\rtimes F : x \mapsto \phi x \phi^{-1} \text{ (conjugation in } Aff(N))
\end{equation*}
is an automorphism of $N\rtimes F$, with $\overline{\mathcal{A}}(F) = F$ and $\overline{\mathcal{A}}(\Gamma) \subseteq \Gamma$. Hence, $A$ is an infra-nil-endomorphism.
\end{theorem}

Let $X$ be a topological space. We write $\Omega(X;x_0)$ the family of all closed paths from $x_0$ to $x_0$. Let $\Omega(X;x_0)/\sim$ be the identifying space with respect to the equivalence relation $\sim$ by homotopty. We write this set
\begin{equation*}
\pi_1(X;x_0)=\Omega(X;x_0)/\sim.
\end{equation*}
The group $\pi_1(X;x_0)$ is called the \emph{fundamental group} at a base point $x_0$ of $X$. If, in particular, $\pi_1(X;x_0)$ is a group consisting of the identity, then $X$ is said to be \emph{simply connected} with respect to a base point $x_0$.

Let $x_0$ and $x_1$ be points in $X$. If there exists a path $w$ joining $x_0$ and $x_1$, then we can define a map
\begin{equation*}
w_\sharp:\Omega(X;x_1)\rightarrow\Omega(X;x_0),\text{ by } w_\sharp(u)=(w\cdot u)\cdot\overline{w},
\end{equation*}
where $u\in \Omega(X;x_1)$, $(w\cdot u)$ is the concatenation of $w$ and $u$ and $\overline{w}$ is $w$ in reverse direction. For $u,v\in \Omega(X;x_1)$ suppose $u\sim v$. Then $w_\sharp(u)\sim w_\sharp(v)$ and thus $w_\sharp$ induces a map
\begin{equation*}
w_*:\pi_1(X;x_1)\rightarrow\pi_1(X;x_0),\text{ by } w_*([u])=[w_\sharp(u)],
\end{equation*}
this map is an isomorphism (see \cite{[Ao-Hi]} Lemma 6.1.4).
\begin{remark}
If $X$ is a path connected space then we can remove the base point and write $\pi_1(X;x_0)=\pi_1(X)$.
\end{remark}

Let $f,g: X \rightarrow Y$ be homotopic and $F$ a homotopy from $f$ to $g$ ($f \sim g \; (F)$). Then for $x_0\in X$ we can define a path $w\in \Omega(Y;f(x_0),g(x_0))$ by 
\begin{equation*}
w(t)=F(x_0,t)\quad t\in[0,1],
\end{equation*}
and the relation between homomorphisms $f_*:\pi_1(X;x_0)\rightarrow\pi_1(Y;f(x_0))$ and $g_*:\pi_1(X;x_0)\rightarrow\pi_1(Y;g(x_0))$ is: $g_*=\overline{w}_*\circ f_*$ (see \cite{[Ao-Hi]} Lemma 6.1.9).\\

Let $X$ and $Y$ be topological spaces and $f: X \rightarrow Y$ a continuous map. Take $x_0 \in X$ and let $y_0 = f(x_0)$. It is clear that $fu = f \circ u \in \Omega(X, y_0)$ for $u\in \Omega(X;x_0)$. Thus we can find a map
\begin{equation*}
f_\sharp:\Omega(X;x_0)\rightarrow\Omega(Y;Y_0),\text{ by } f_\sharp(u)=fu,
\end{equation*}
where $u\in \Omega(X;x_0)$. If $u \sim v \; (F)$ for $u,v\in \Omega(X;x_0)$, then we have $fu \sim fv \; (f\circ F)$, from which the following map will be induced:
\begin{equation*}
f_*:\pi_1(X;x_0)\rightarrow\pi_1(Y;y_0),\text{ by } f_*([u])=[f_\sharp(u)]=[fu],
\end{equation*}
It is easy to check that $f_*$ is a homomorphism. We say that $f_*$ is a \emph{homomorphism induced from a continuous map} $f: X \rightarrow Y$.

\begin{lemma}[\cite{[Ao-Hi]} Remark 6.7.9]\label{Ao Remark 6.7.9}
Let $f,g:N/\Gamma \rightarrow N/\Gamma$ be continuous maps of a nil-manifold and let $f(x_0) = g(x_0)$ for some $x_0 \in N/\Gamma$. Then $f$ and $g$ are homotopic if and only if $f_* = g_* : \pi_1 (N/\Gamma,x_0) \rightarrow \pi_1 (N/\Gamma,f(x_0))$.
\end{lemma}

\begin{theorem}[\cite{[Ao-Hi]} Theorem 6.3.4]\label{[Ao-Hi] Theorem 6.3.4}
If $\pi:Y\rightarrow X$ is the universal covering, then for each $b\in Y$
\begin{enumerate}
\item[(1)]the map $\alpha\mapsto \alpha(b)$ is a bijection from $G(\pi)$ onto $\pi^{-1}(\pi(b))$,
\item[(2)]the map $\psi_b:G(\pi)\rightarrow \pi_1(X,\pi(b))$ by $\alpha\mapsto [\pi\circ u_{\alpha(b)}]$ is an isomorphism where $u_{\alpha(b)}$ is a path from $b$ to $\alpha(b)$.
\end{enumerate}
Furthermore, the action of $G(\pi)$ on $Y$ is properly discontinuous and $Y/G(\pi)$ is homeomorphic to $X$ .
\end{theorem}

\begin{theorem}[\cite{[Ao-Hi]} Theorem 6.3.7]\label{[Ao-Hi] Theorem 6.3.7}
Let $G$ be a group and $X$ a topological space. Suppose that $G$ acts on $X$ and the action is properly discontinuous. Then
\begin{enumerate}
\item[(1)]the natural projection $\pi:X\rightarrow X/G$ is a covering map,
\item[(2)]if $X$ is simply connected, then the fundamental group $\pi_1(X/G)$ is isomorphic
to $G$.
\end{enumerate}
\end{theorem}

\begin{cor}\label{isomorphismGamma}
Let $N/\Gamma$ be an infra-nil-manifold and $\pi:N\rightarrow N/\Gamma$ be the natural projection. Then
\begin{equation*}
\Gamma\cong\pi_1(N/\Gamma)\cong G(\pi).
\end{equation*}
\end{cor}
\begin{proof}
Since $\Gamma$ acts on $N$ properly discontinuous, the natural projection $\pi:N\rightarrow N/\Gamma$ is a covering map. Since $N$ is simply connected, by Theorem \ref{[Ao-Hi] Theorem 6.3.7} we have $\Gamma\cong\pi_1(N/\Gamma)$.

On the other hand, since $N$ is simply connected and $\Gamma$ acts on $N$ properly discontinuous the natural projection $\pi:N\rightarrow N/\Gamma$ is the universal covering map. So by Theorem \ref{[Ao-Hi] Theorem 6.3.4} we have $\Gamma\cong G(\pi)$.
\end{proof}

From now on we only consider $N/\Gamma$ as a nil-manifold.

\begin{lemma}\label{fbar=Abar on gamma}
Let $f:N/\Gamma \rightarrow N/\Gamma$ be a continuous map of a nil-manifold, and $A:N/\Gamma \rightarrow N/\Gamma$ be the unique nil-endomorphism homotopic to $f$, then $\overline{f}_*=\overline{A}_*:\Gamma\rightarrow\Gamma$.
\end{lemma}
\begin{proof}
By corollary \ref{isomorphismGamma}, $\overline{f}_*$ and $\overline{A}_*$ are two maps on $\Gamma$. For $[e]=\{x\in N:\gamma(x)=\gamma.x=e\text{ for some }\gamma\in\Gamma\}$, we have
\begin{equation*}
f([e])=f\circ \pi (e)=\pi\circ \overline{f}(e)=\pi(\overline{f}(e))=\pi(e)=[e]=A([e]).
\end{equation*}
So according to lemma \ref{Ao Remark 6.7.9}, $\overline{f}_*=\overline{A}_*$.
\end{proof}

\begin{lemma}[\cite{[Su 2]} Lemma 1.3]\label{Su 2 Lemma 1.3}
Let $f:N/\Gamma \rightarrow N/\Gamma$ be a self-covering map of a nil-manifold and $A:N/\Gamma\rightarrow N/\Gamma$ denote the nil-endomorphism homotopic to $f$. If $f$ is a TA-covering map, then $A$ is hyperbolic.
\end{lemma}

\begin{lemma}[\cite{[Su 2]} Lemma 1.5]\label{fixedpoint}
Let $f:N/\Gamma\rightarrow N/\Gamma$ be a self-covering map and let $\overline{f}$ : $N\rightarrow N$ be a lift of $f$ by the natural projection $\pi:N\rightarrow N/\Gamma$. If $f$ is a TA-covering map then $\overline{f}$ has exactly one fixed point.
\end{lemma}

For continuous maps $f$ and $g$ of $N$ we define $D(f, g)= \sup\{D(f(x), g(x)) : x\in N\}$ where $D$ denotes a left invariant, $\Gamma$-invariant Riemannian distance for $N$ . Notice that $D(f, g)$ is not necessary finite.

Suppose that $f$ : $N/\Gamma\rightarrow N/\Gamma$ is a TA-covering map. Let $A:N/\Gamma\rightarrow N/\Gamma$ be the nil-endomorphism homotopic to $f$, and let $\overline{A}$ : $N \rightarrow N$ be the automorphism which is a lift of $A$ by the natural projection $\pi$. Since $D_e\overline{A}$ is hyperbolic by Lemma \ref{Su 2 Lemma 1.3}, the Lie algebra $Lie(N)$ of $N$ splits into the direct sum $Lie(N)=E_e^s\oplus E_e^u$ of subspaces $E_e^s$ and $E_e^u$ such that $D_e\overline{A}(E_e^s)=E_e^s$, $D_e\overline{A}(E_e^u)=E_e^u$ and there are $c>1,0<\lambda<1$ so that for all $n\geq 0$
\begin{align*}
&||D_e\overline{A}^n(v)||\leq c\lambda^n||v|| \quad (v\in E_e^s), \\
&||D_e\overline{A}^{-n}(v)||\leq c\lambda^n||v|| \quad (v\in E_e^u),
\end{align*}
where $||\cdot||$ is the Riemannian metric. Let $\overline{L}^{\sigma}(e)=\exp(E_e^{\sigma})\;(\sigma=s, u)$ and let
$\overline{L}^{\sigma}(x)=x\cdot\overline{L}^{\sigma}(e)(\sigma=s, u)$ for $x\in N$ . Since left translations are isometries under the metric $D$, it follows that for all $x\in N$
\begin{align*}
&\overline{L}^s(x)=\{y\in N : D(\overline{A}^i(x),\overline{A}^i(y))\rightarrow 0\;(i\rightarrow\infty)\},\\
&\overline{L}^u(x)=\{y\in N : D(\overline{A}^i(x),\overline{A}^i(y))\rightarrow 0\;(i\rightarrow -\infty)\} .
\end{align*}

\begin{lemma}[\cite{[Hi]} Lemma 3.2]
For $x,y\in N$, $L^s(X)\cap{L^u(y)}$ consists of exactly one point.
\end{lemma}
For $x,y\in N$ denote as $\beta(x, y)$ the point in $L^s(X)\cap{L^u(y)}$ .
\begin{lemma}[\cite{[Hi]} Lemma 3.2]
(1) For $L>0$ and $\epsilon>0$ there exists $J>0$ such that for $x,y\in N$ if $D(\overline{A}^i(x),\overline{A}^i(y))\leq L$ for all $i$ with $|i|\leq J$ , then $D(x, y)\leq\epsilon$.\\
(2) For given $L>0$, if $D(\overline{A}^i(x),\overline{A}^i(y))\leq L$ for all $i\in \mathbb{Z}$ , then $x=y \;(x, y\in N)$.
\end{lemma}

\begin{lemma}[\cite{[Su 2]} Lemma 2.3]\label{Semiconjugacy}
Under the assumptions and notations as above, there is a unique map $\overline{h}$ : $N\rightarrow N$ such that
\begin{enumerate}
\item[(1)]$\overline{A}\circ\overline{h}=\overline{h}\circ\overline{f}$,
\item[(2)]$D(\overline{h}, id_{N})$ is finite,
\end{enumerate}
where $id_{N}$ : $N\rightarrow N$ is the identity map of $N$. Furthermore $\overline{h}$ is surjective and uniformly continuous under $D$ .

In addition, if $f$ is not an expanding map then $\overline{h}$ is a homeomorphism i.e. $\overline{h}$ is $D$-biuniformly continuous. (See \cite{[Ao-Hi]} Proposition 8.4.2)
\end{lemma}

\begin{lemma}[\cite{[Su 2]} Lemma 2.4]\label{lm Su 2 Lemma 2.4}
For the semiconjugacy $\overline{h}$ of lemma \ref{Semiconjugacy}, we have the following properties:
\begin{enumerate}
\item[(1)]There exists $K>0$ such that $D(\overline{h}(x \gamma), \overline{h}(x) \gamma)<K$ for $x\in \mathbb{N}$ and $\gamma\in\Gamma$.
\item[(2)]For any $\lambda>0$ , there exists $L\in \mathbb{N}$ such that $D(\overline{h}(x \gamma), \overline{h}(x) \gamma)<\lambda$ for $x\in N$ and $\gamma\in\overline{A}_*^L(\Gamma)$ .
\item[(3)]For $x\in N$ and $\gamma\in\bigcap_{i=0}^{\infty}\overline{A}_*^i(\Gamma)$ , we have $\overline{h}(x \gamma)=\overline{h}(x) \gamma$ .
\item[(4)]For $x\in N$ and $\gamma\in\Gamma$, we have $\overline{h}(x \gamma)\in \overline{L}^s(\overline{h}(x) \gamma)$.
\end{enumerate}
\end{lemma}

\begin{remark} \label{rm page 270 (8.5)}
By part (2) of theorem \ref{Semiconjugacy}, there is a $\delta_K>0$ such that $D(\overline{h}(x),x)<\delta_K$ for $x\in N$, we have (see \cite{[Ao-Hi]} page 270 (8.5))
\begin{equation*}
\overline{W}^s(x)\subset U_{\delta_K}(\overline{L}^s(\overline{h}(x)))\quad\text{and}\quad \overline{W}^u(x;\textbf{e})\subset U_{\delta_K}(\overline{L}^u(\overline{h}(x))).
\end{equation*}
\end{remark}

By lemma \ref{fixedpoint} if $f:N/\Gamma \rightarrow N/\Gamma$ is a $TA$-map and $\overline{f}:N\rightarrow N$ a lift of it, then there exists a unique fixed point say $b$ such that $\overline{f}(b)=b$. For simplisity we can suppose that $b=e$. Indeed, we can choose a homeomorphism $\varphi$ of $N$ such that $\varphi(\pi(b))=e$. Then $\varphi\circ f\circ\varphi^{-1}$ is a $TA$-covering map such that $\varphi\circ f\circ\varphi^{-1}(e)=e$.\\

Let $x\in N$, we define the stable set and unstable sets of $x$ for $f$ and $A$ as follow (for more details see \cite{[Ao-Hi]}):
\begin{align*}
\overline{W}^s(x)&=\{y\in N:\lim_{i\rightarrow\infty}D(\overline{f}^i(x),\overline{f}^i(y))=0\},\\
\overline{W}^u(x,\textbf{e})&=\{y\in N:\lim_{i\rightarrow -\infty}D(\overline{f}^i(x),\overline{f}^i(y))=0\},
\end{align*}
Where $\textbf{e}=(\ldots,e,e,e,\ldots)$.

\begin{remark}\label{rm w l}
By lemma \ref{Semiconjugacy}, since $\overline{h}$ is $D$-uniformly continuous then $\overline{h}(\overline{W}^s(x))=\overline{L}^s(\overline{h}(x))$ and $\overline{h}(\overline{W}^u(x;\textbf{e}))=\overline{L}^u(\overline{h}(x))$.
\end{remark}

\begin{lemma}\label{lm gamma in w}
The following statements hold:
\begin{enumerate}
\item[(1)]$\overline{W}^s(x) \gamma=\overline{W}^s(x \gamma)$ for $\gamma\in\Gamma$ and $x\in N$,
\item[(2)]$\overline{W}^u(x;\textbf{e}) \gamma=\overline{W}^u(x \gamma;\textbf{e})$ for $\gamma\in\Gamma$ and $x\in N$,
\end{enumerate}
\end{lemma}
\begin{proof}
It is an easy corollary of lemma 6.6.11 of \cite{[Ao-Hi]}. According to corollary \ref{isomorphismGamma}, in the mentioned lemma put $\Gamma$ instead of $G(\pi)$ and $N$ instead of $\overline{X}$.
\end{proof}

\begin{lemma}\label{lm gamma in l}
The following statements hold:
\begin{enumerate}
\item[(1)]$\overline{L}^s(x)\gamma=\overline{L}^s(x \gamma)$ for $\gamma\in\Gamma$ and $x\in N$,
\item[(2)]$\overline{L}^u(x)\gamma=\overline{L}^u(x \gamma)$ for $\gamma\in\Gamma$ and $x\in N$,
\end{enumerate}
\end{lemma}
\begin{proof}
Proof is the same as in lemma \ref{lm gamma in w}.
\end{proof}

\begin{lemma}[\cite{[Su 2]} Lemma 5.4]
Let $N/\Gamma$ be a nil-manifold. If $f:N/\Gamma\rightarrow N/\Gamma$ is a $TA$-covering map, then the nonwandering set $\Omega(f)$ coincides with the entire space $N/\Gamma$.
\end{lemma}

\begin{lemma}[\cite{[Ao-Hi]} Lemma 8.6.2] \label{Ao lm 8.6.2}
For $\epsilon>0$ there is $\delta>0$ such that if $D(x,y)<\delta, \; x,y\in N$ then $\overline{W}^s(x)\subset U_{\epsilon}(\overline{W}^s(y))$ and $\overline{W}^u(x;\textbf{e})\subset U_{\epsilon}(\overline{W}^u(y;\textbf{e}))$. Where for a set $S,\; U_{\epsilon}(S)=\{y\in N:D(y,S)<\epsilon\}$.
\end{lemma}

\section{Proof of Main Theorem}
In this section we suppose that $f:N/\Gamma \rightarrow N/\Gamma$ is a special $TA$-covering map of a nil-manifold which is not injective or expanding, and $A:N/\Gamma \rightarrow N/\Gamma$ is the unique nil-endomorphism homotopic to $f$.\\

\textbf{Sketch of proof.} By lemma \ref{Semiconjugacy}, there is a unique semiconjugacy $\overline{h}:N\rightarrow N$ between $\overline{f}$ and $\overline{A}$, such that by proposition \ref{lm su1 9}.(3), $\overline{h}(v\gamma)=\overline{h}(v)\gamma$, for each $\gamma \in \overline{W}^u(e;\textbf{e})\cap\Gamma$ and $v\in\overline{W}^u(e;\textbf{e})$. Through proposition \ref{properties of h'} to proposition \ref{final lemma} we show that for all $\gamma\in\Gamma$ and $x\in N$, $\overline{h}(x \gamma)=\overline{h}(x)\gamma$. Based on this result,  $\overline{h}$ induces a homeomorphism $h:N/\Gamma\rightarrow N/\Gamma$ which is the conjugacy between $f$ and $A$.\\

To prove the main theorem we need some consequential lemmas and propositions.

\begin{lemma} \label{lm x in}
The following statements hold:
\begin{enumerate}
\item[(1)]Let $D$ be the metric of $N$ as above. for each $x\in N$, $D(x^{-1},e)=D(x,e)$.
\item[(2)]If $x\in\overline{W}^u(e;\textbf{e})$, then $\overline{W}^u(x;\textbf{e})=\overline{W}^u(e;\textbf{e})$.
\item[(3)]If $x\in\overline{L}^u(e)$, then $\overline{L}^u(x)=\overline{L}^u(e)$.
\end{enumerate}
\end{lemma}
\begin{proof}
(1) \begin{align*}
D(x^{-1},e)&=D(x^{-1}e,x^{-1}x)\\ \text{($D$ is left invariant)} \quad &=D(e,x)\\&=D(x,e)
\end{align*}
(2) Since $x\in\overline{W}^u(e;\textbf{e})$, we have $D(\overline{f}^i(x),\overline{f}^i(e))\rightarrow 0$ as $i\rightarrow -\infty$. Let $y\in\overline{W}^u(e;\textbf{e})$, then $D(\overline{f}^i(y),\overline{f}^i(e))\rightarrow 0$ as $i\rightarrow -\infty$. We have,
\begin{equation*}
D(\overline{f}^i(y),\overline{f}^i(x))<D(\overline{f}^i(y),\overline{f}^i(e))+D(\overline{f}^i(e),\overline{f}^i(x))\rightarrow 0 \text{ as } i\rightarrow -\infty.
\end{equation*}
So, $y\in \overline{W}^u(x;\textbf{e})$, i.e. $\overline{W}^u(e;\textbf{e})\subset\overline{W}^u(x;\textbf{e})$.
Conversely, if $y\in\overline{W}^u(x;\textbf{e})$ then $D(\overline{f}^i(y),\overline{f}^i(x))\rightarrow 0$ as $i\rightarrow -\infty$, and
\begin{equation*}
D(\overline{f}^i(y),\overline{f}^i(e))<D(\overline{f}^i(y),\overline{f}^i(x))+D(\overline{f}^i(x),\overline{f}^i(e))\rightarrow 0 \text{ as } i\rightarrow -\infty.
\end{equation*}
So, $y\in \overline{W}^u(e;\textbf{e})$, i.e. $\overline{W}^u(x;\textbf{e})\subset\overline{W}^u(e;\textbf{e})$.\\
(3) Its proof is the same as part (2).
\end{proof}

For simplicity, let $\Gamma_{\overline{f}}=\overline{W}^u(e;\textbf{e})\cap\Gamma$ and $\Gamma_{\overline{A}}=\overline{L}^u(e)\cap\Gamma$.

\begin{prop}\label{lm su1 9}
The following statements hold:
\begin{enumerate}
\item[(1)]$\Gamma_{\overline{A}}$ and $\Gamma_{\overline{f}}$ are subgroups of $\Gamma$.
\item[(2)]$\Gamma_{\overline{f}} \subset  \Gamma_{\overline{A}}$.
\item[(3)]$\overline{h}(v\gamma)=\overline{h}(v)\gamma$, for each $\gamma\in\Gamma_{\overline{f}}$ and $v\in\overline{W}^u(e;\textbf{e})$.
\item[(4)] If $\overline{W}^u(\gamma_1;\textbf{e})=\overline{W}^u(\gamma_2;\textbf{e})$, for some $\gamma_1,\gamma_2\in\Gamma$, then we have
\begin{equation*}
\overline{h}(x \gamma_1^{-1})\gamma_1=\overline{h}(x \gamma_2^{-1})\gamma_2, \text{ for } x\in \overline{W}^u(\gamma_1;\textbf{e}).
\end{equation*}
\end{enumerate}
\end{prop}

\begin{proof}
(1) Let $\gamma_1,\gamma_2\in \Gamma_{\overline{A}}$. Since $\Gamma$ is a group we have $\gamma_1\gamma_2^{-1}\in\Gamma$.  Now consider that $\gamma_1,\gamma_2\in \overline{L}^u(e)$, since $A^i(e)=e$, for all $i$, then by definition,
\begin{align} \label{eq Dleft}
\lim_{i\rightarrow -\infty}D(A^i(\gamma_1),e)&=\lim_{i\rightarrow -\infty}D(A^i(\gamma_1),A^i(e))=0\nonumber \\
\lim_{i\rightarrow -\infty}D(A^i(\gamma_2),e)&=\lim_{i\rightarrow -\infty}D(A^i(\gamma_2),A^i(e))=0.
\end{align}
As $D$ is left invariant we have 
\begin{align*}
0\leq\lim_{i\rightarrow -\infty} D(A^i(\gamma_1 \gamma_2^{-1}),A^i(e))&=\lim_{i\rightarrow -\infty} D(A^i(\gamma_1)A^i(\gamma_2^{-1}),e)\\
&=\lim_{i\rightarrow -\infty} D(A^i(\gamma_1)A^{-i}(\gamma_2),A^i(\gamma_1)A^{-i}(\gamma_1))\\
\text{(D is left invariant)}\quad &=\lim_{i\rightarrow -\infty} D(A^{-i}(\gamma_2),A^{-i}(\gamma_1))\\
&\leq\lim_{i\rightarrow -\infty} D(A^{-i}(\gamma_2),e)+D(e,A^{-i}(\gamma_1))\\
\text{(Lemma \ref{lm x in}.(1) and \eqref{eq Dleft})}\quad &=\lim_{i\rightarrow -\infty} D(A^i(\gamma_2),e)+D(A^i(\gamma_1),e)=0
\end{align*}
Thus $\gamma_1\gamma_2^{-1}\in \overline{L}^u(e)$ and $\overline{L}^u(e)$ is a subgroup of $N$. So $\overline{L}^u(e)\cap\Gamma$ is a subgroup of $\Gamma$.

For the second part, Let $\gamma_1,\gamma_2\in \Gamma_{\overline{f}}$. Since $\Gamma$ is a group we have $\gamma_1\gamma_2^{-1}\in\Gamma$. Now consider that $\gamma_1,\gamma_2\in \overline{W}^u(e;\textbf{e})$. Then,
\begin{align*}
&\overline{W}^u(e;\textbf{e})\gamma_1\\ 
\text{(Lemma \ref{lm gamma in w})}\quad &=\overline{W}^u(e \gamma_1;\textbf{e})\\
&=\overline{W}^u(\gamma_1;\textbf{e})\\
\text{(Lemma \ref{lm x in}) (2)}\quad &=\overline{W}^u(e;\textbf{e}).
\end{align*}
Similarly, $\overline{W}^u(e;\textbf{e})\gamma_2=\overline{W}^u(e;\textbf{e})$. So we have $\overline{W}^u(e;\textbf{e})\gamma_1=\overline{W}^u(e;\textbf{e})\gamma_2$ and then $\gamma_1\gamma_2^{-1}\in \overline{W}^u(e;\textbf{e})$, and we have the result.

(2) Take $\gamma\in \Gamma_{\overline{f}}$, such that $\gamma\notin \Gamma_{\overline{A}}$. So, $\gamma \notin \overline{L}^u(e)$ and for each $n\in \mathbb{Z}, \; n\neq 0$, $\gamma^n \notin \overline{L}^u(e)$. On the other hand, by part (1), remark \ref{rm page 270 (8.5)} and the fact that $\overline{h}(e)=e$, for all $n\in \mathbb{Z}$, we have $\gamma^n\in \overline{W}^u(e;\textbf{e})\subset U_{\delta_K}(\overline{L}^u(e))$, which is impossible.

(3)Let $\gamma\in \Gamma_{\overline{f}}$ and $v\in\overline{W}^u(e;\textbf{e})$. We have
\begin{align*}
v \gamma & \in \overline{W}^u(e;\textbf{e})\gamma\\
\text{(Lemma \ref{lm gamma in w})}\quad &=\overline{W}^u(e \gamma;\textbf{e})\\
&=\overline{W}^u(\gamma;\textbf{e})\\
\text{(Lemma \ref{lm x in} (2))}\quad &=\overline{W}^u(e;\textbf{e}),
\end{align*}
so,
\begin{align*}
\overline{h}(v \gamma)&\in \overline{h}(\overline{W}^u(e;\textbf{e}))\\
\text{(Remark \ref{rm w l})}\quad &=\overline{L}^u(e).
\end{align*}
By part (2), $\gamma\in \Gamma_{\overline{A}}$, Thus
\begin{align*}
\overline{h}(v) \gamma &\in \overline{h}(\overline{W}^u(e;\textbf{e})) \gamma\\
\text{(Remark \ref{rm w l})}\quad &=\overline{L}^u(e)\gamma\\
\text{(Lemma \ref{lm gamma in l})}\quad &=\overline{L}^u(e \gamma)\\
&=\overline{L}^u(\gamma)\\
\text{(Lemma \ref{lm x in} (3))}\quad &=\overline{L}^u(e).
\end{align*}
Again by Lemma \ref{lm x in} (3) and last part of the above relation, $\overline{L}^u(\overline{h}(v)\gamma)=\overline{L}^u(e)$, and 
\begin{equation*}
\overline{h}(v\gamma)\in\overline{L}^u(e)=\overline{L}^u(\overline{h}(v)\gamma).
\end{equation*}
On the other hand, by part (4) of lemma \ref{lm Su 2 Lemma 2.4}, $\overline{h}(v\gamma)\in \overline{L}^s(\overline{h}(v)\gamma)$. Since $\overline{L}^u(\overline{h}(v)\gamma)\cap\overline{L}^s(\overline{h}(v)\gamma)=\{\overline{h}(v)\gamma\}$ (see \cite{[Su 2]} Lemma 2.1), then $\overline{h}(v\gamma)=\overline{h}(v)\gamma$.

(4) Let $x\in \overline{W}^u(\gamma_1;\textbf{e})=\overline{W}^u(\gamma_2;\textbf{e})$. For $\gamma_1,\gamma_2 \in \Gamma$, we have $\gamma_2\in \overline{W}^u(\gamma_2;\textbf{e})= \overline{W}^u(\gamma_1;\textbf{e})=\overline{W}^u(e;\textbf{e})\gamma_1$. Thus, $\gamma_2\gamma_1^{-1}\in \overline{W}^u(e;\textbf{e})$, and then $\gamma_2\gamma_1^{-1}\in \Gamma_{\overline{f}}$. Similarly, $x\gamma_1^{-1},x\gamma_2^{-1}\in \overline{W}^u(e;\textbf{e})$. Now, by part (3),
\begin{align*}
\overline{h}(x\gamma_1^{-1})\gamma_1&=\overline{h}(x\gamma_2^{-1}\gamma_2\gamma_1^{-1})\gamma_1\\
&=\overline{h}(x\gamma_2^{-1})\gamma_2\gamma_1^{-1}\gamma_1\\
&=\overline{h}(x\gamma_2^{-1})\gamma_2.
\end{align*}
\end{proof}

According to part (4) of proposition \ref{lm su1 9}, we can define a map $\overline{h}':\bigcup_{\gamma\in \Gamma}\overline{W}^u(\gamma;\textbf{e})\rightarrow \bigcup_{\gamma\in \Gamma}\overline{L}^u(\gamma)$, by
\begin{equation*}
\overline{h}'(x)=\overline{h}(x\gamma^{-1})\gamma\quad x\in \overline{W}^u(\gamma;\textbf{e})\: (\gamma\in\Gamma).
\end{equation*}

Next lemma shows some properties of $\overline{h}'$:
\begin{prop}\label{properties of h'}
The following statements hold:
\begin{enumerate}
\item[(1)]$\overline{A}\circ\overline{h}'=\overline{h}'\circ\overline{f}$ on $\bigcup_{\gamma\in \Gamma}\overline{W}^u(\gamma;\textbf{e})$,
\item[(2)]$D(\overline{h}',id_{|\bigcup_{\gamma\in \Gamma}\overline{W}^u(\gamma;\textbf{e})})<\infty$,
\item[(3)]$\overline{h}'(\gamma)=\gamma$ for $\gamma\in\Gamma$,
\item[(4)]if $x\in \overline{W}^u(\gamma;\textbf{e})\: (\gamma\in\Gamma)$, then $\overline{h}'(x)\in\overline{L}^u(\gamma)$ and $\overline{h}'(x)\in\overline{L}^s(\overline{h}(x))$,
\item[(5)]if $y\in\overline{W}^s(x)$ for $x,y\in\bigcup_{\gamma\in \Gamma}\overline{W}^u(\gamma;\textbf{e})$, then $\overline{h}'(y)\in\overline{L}^s(\overline{h}'(x))$.
\end{enumerate}
\end{prop}
\begin{proof}
(1) Suppose that $x\in\overline{W}^u(\gamma;\textbf{e})=\overline{W}^u(e;\textbf{e})\gamma$, for some $\gamma\in\Gamma$. Then
\begin{equation}\label{eq inverse gamma}
x\gamma^{-1}\in \overline{W}^u(e;\textbf{e}).
\end{equation}
By \cite{[Ao-Hi]} page 205, we have $\overline{f}\big(\overline{W}^u(\gamma;\textbf{e})\big)=\overline{W}^u(\overline{f}(\gamma);\textbf{e})$. Here $\overline{f}(\gamma)$ means $\overline{f}_*(\gamma)$ which by lemma \ref{fbar=Abar on gamma} is equal to $\overline{A}_*(\gamma)$ and $\overline{A}_*(\gamma)\in\Gamma$. Therefore,
\begin{equation*}
\overline{f}(x)\in\overline{f}\big(\overline{W}^u(\gamma;\textbf{e})\big)=\overline{W}^u(\overline{A}_*(\gamma);\textbf{e})=\overline{W}^u(e;\textbf{e})\overline{A}_*(\gamma) 
\end{equation*}
so,
\begin{equation}\label{eq fbar x}
(\overline{f}(x))(\overline{A}_*(\gamma))^{-1}\in\overline{W}^u(e;\textbf{e}).
\end{equation}
Thus we have
\begin{align*}
\overline{A}\circ\overline{h}'(x)&=\overline{A}\big(\overline{h}(x\gamma^{-1})\gamma\big)\\
\text{(\eqref{eq inverse gamma} and proposition \ref{lm su1 9}.(3))}\quad &=\overline{A}\big(\overline{h}(x\gamma^{-1}\gamma)\big)\\
&=\overline{A}\circ \overline{h}(x)\\
\text{(lemma \ref{Semiconjugacy})}\quad &=\overline{h}\circ \overline{f}(x)\\
&=\overline{h}\bigg((\overline{f}(x))\big((\overline{A}_*(\gamma))^{-1}\big)(\overline{A}_*(\gamma))\bigg)\\
\text{(\eqref{eq fbar x} and proposition \ref{lm su1 9}.(3))}\quad &=\overline{h}\bigg((\overline{f}(x))\big((\overline{A}_*(\gamma))^{-1}\big)\bigg)(\overline{A}_*(\gamma))\\
&=\overline{h}'(\overline{f}(x))\\
&=\overline{h}'\circ\overline{f}(x).
\end{align*}

(2) Let $x\in\overline{W}^u(\gamma;\textbf{e})$, for some $\gamma\in\Gamma$, and let $\delta_K>0$ be satisfying $D(\overline{h},id_N)<\delta_K$. Then
\begin{align*}
D(\overline{h}'(x),x)&=D\big(\overline{h}(x\gamma^{-1})\gamma,x\big)\\ &=D\big(\overline{h}(x\gamma^{-1})\gamma,x\gamma^{-1}\gamma\big)\\
\text{(D is $\Gamma$-invariant)}\quad &=D\big(\overline{h}(x\gamma^{-1}),x\gamma^{-1}\big)\\
&<\delta_K.
\end{align*}

(3) For any $\gamma\in\Gamma$, by definition we have
\begin{equation*}
\overline{h}'(\gamma)=\overline{h}(\gamma\gamma^{-1})\gamma=\overline{h}(e)\gamma=e\gamma=\gamma.
\end{equation*}

(4) Let $x\in\overline{W}^u(\gamma;\textbf{e})$, for some $\gamma\in\Gamma$. We have
\begin{align*}
\overline{h}'(x)=\overline{h}(x\gamma^{-1})\gamma&\in \overline{h}(\overline{W}^u(\gamma;\textbf{e})\gamma^{-1})\gamma\\
\text{(lemma \ref{lm gamma in w})}\quad
&=\overline{h}(\overline{W}^u(e;\textbf{e})\gamma\gamma^{-1})\gamma\\ &=\overline{h}(\overline{W}^u(e;\textbf{e}))\gamma\\
\text{(remark \ref{rm w l})}\quad &=\overline{L}^u(e)\gamma\\
\text{(lemma \ref{lm gamma in l})}\quad
&=\overline{L}^u(\gamma),
\end{align*}
and
\begin{align*}
\overline{h}'(x)&=\overline{h}(x\gamma^{-1})\gamma\\
\text{(lemma \ref{lm Su 2 Lemma 2.4}.(4))}\quad
&\in \overline{L}^s(\overline{h}(x)\gamma^{-1})\gamma\\
\text{(lemma \ref{lm gamma in l}.(1))}\quad &=\overline{L}^s(\overline{h}(x))\gamma^{-1}\gamma\\
&=\overline{L}^s(\overline{h}(x)).
\end{align*}

(5) By the second part of proof of (4), we have
\begin{equation*}
\overline{L}^s(\overline{h}'(y))=\overline{L}^s\big(\overline{h}(y)\big)=\overline{h}(\overline{W}^s(y))=\overline{h}(\overline{W}^s(x))=\overline{L}^s\big(\overline{h}(x)\big)=\overline{L}^s(\overline{h}'(x)),
\end{equation*}
so, $\overline{h}'(y)\in \overline{L}^s(\overline{h}'(x))$.
\end{proof}

\begin{lemma}[\cite{[Su 2]} Lemma 7.6]
For each $u,v\in N$, $\overline{W}^u(u;\textbf{e})\cap\overline{W}^s(v)$ is the set of one point.
\end{lemma}

According to the above lemma, define $\overline{\iota}(u,v)=\overline{W}^u(u;\textbf{e})\cap\overline{W}^s(v)$.

\begin{lemma}\label{su1 lm11}
For $\epsilon>0$, there is $\delta>0$ such that
\begin{equation*}
D(u,v)<\delta\Rightarrow max\{D(\overline{\iota}(u,v),u),D(\overline{\iota}(u,v),v)\}<\epsilon\quad(u,v\in N)
\end{equation*}
\end{lemma}
\begin{proof}
Let $\epsilon>0$ be given. Since $\overline{h}$ is $D$-biuniformly contiuous there exists $\epsilon'>0$ such that
\begin{equation*}
D(x,y)<\epsilon'\Rightarrow D(\overline{h}^{-1}(x),\overline{h}^{-1}(y))<\epsilon\quad (x,y\in N).
\end{equation*}
By \cite{[Ao-Hi]} theorem 6.6.5 or \cite{[Su 2]} lemma 7.2, since $N$ is simply connected, $\overline{A}$ has local product structure (for definition and details, see \cite{[Ao-Hi]}), and then for $\epsilon>0$ there exists $\delta'>0$ such that
\begin{equation*}
D(u,v)<\delta'\Rightarrow max\{D(\beta(u,v),u),D(\beta(u,v),v)\}<\epsilon'\quad(u,v\in N)
\end{equation*}
Again since $\overline{h}$ is $D$-biuniformly continuous, there exists $\delta>0$ such that
\begin{equation*}
D(u,v)<\delta\Rightarrow D(\overline{h}(u),\overline{h}(v))<\delta'\quad (u,v\in N).
\end{equation*}
We know that $\overline{h}\big(\overline{\iota}(u,v)\big)=\beta(\overline{h}(u),\overline{h}(v))$ therefore
\begin{align*}
D(u,v)<\delta &\Rightarrow D(\overline{h}(u),\overline{h}(v))<\delta'\\
&\Rightarrow max\{D(\beta(\overline{h}(u),\overline{h}(v)),\overline{h}(u)),D(\beta(\overline{h}(u),\overline{h}(v)),\overline{h}(v))\}<\epsilon'\\
&\Rightarrow max\{D(\overline{h}(\overline{\iota}(u,v)),\overline{h}(u)),D(\overline{h}(\overline{\iota}(u,v)),\overline{h}(v))<\epsilon'\\
&\Rightarrow max\{D(\overline{\iota}(u,v),u),D(\overline{\iota}(u,v),v)\}<\epsilon.
\end{align*}
\end{proof}

\begin{prop}\label{su1 lm12}
$\overline{h}'$ is $D$-uniformly continuous.
\end{prop}
\begin{proof}
Suppose that the statement is false. So there is $\epsilon_0>0$, for all $\delta>0$, there are $x,y\in \bigcup_{\gamma\in \Gamma}\overline{W}^u(\gamma;\textbf{e})$ such that
\begin{equation}\label{eq1 lm12}
D(x,y)<\delta \text{ and } D(\overline{h}'(x),\overline{h}'(y))>\epsilon_0.
\end{equation}
By definition of $\overline{L}^{\sigma}(x)$ $(x\in N,\;\sigma=s,u)$, for $w\in\overline{L}^s(v)$ there is $\epsilon_1>0$ such that
\begin{equation}\label{eq2 lm12}
D(v,w)<\epsilon_0/2 \Rightarrow D(\overline{L}^u(v),\overline{L}^u(w))>\epsilon_1.
\end{equation}
Take $n>0$ and $\delta_1>0$ such that $\epsilon^n\geq 2\delta_K$ and $\delta_1^n\leq 2\delta_K$.

By Lemma \ref{Ao lm 8.6.2}, there exists $\delta_2>0$ such that
\begin{equation}\label{eq3 lm12}
D(v,w)<\delta_2 \Rightarrow \overline{W}^u(w,\textbf{e})\subset U_{\delta_1}\big(\overline{W}^u(v,\textbf{e})\big).
\end{equation}

Since $\overline{h}$ is continuous, take $\delta_3>0$ such that
\begin{equation}\label{eq4 lm12}
D(u,v)<\delta_3 \Rightarrow D(\overline{h}(u),\overline{h}(v))<\epsilon_0/2.
\end{equation}

By lemma \ref{su1 lm11}, there is $0<\delta<\delta_2$ such that
\begin{equation}\label{eq5 lm12}
D(x,y)<\delta \Rightarrow D(y\gamma_y^{-1},\overline{\iota}(y,x)\gamma_y^{-1})=
D(y,\overline{\iota}(y,x))<\delta_3\quad(x,y\in N).
\end{equation}

Now consider $x,y\in \bigcup_{\gamma\in \Gamma}\overline{W}^u(\gamma;\textbf{e})$ satisfy \eqref{eq1 lm12}. There exist $\gamma_x,\gamma_y\in \Gamma$ such that $x\in\overline{W}^u(\gamma_x,\textbf{e})$ and $y\in\overline{W}^u(\gamma_y,\textbf{e})$. We have

\begin{align} \label{eq6 lm12}
D\big(\overline{h}'(x),\overline{h}'(\overline{\iota}(y,x))\big) &\geq D\big(\overline{h}'(x),\overline{h}'(y)\big) - D\big(\overline{h}'(y),\overline{h}'(\overline{\iota}(y,x))\big) \nonumber \\
\text{(by \eqref{eq1 lm12})}\quad &\geq \epsilon_0 - D\big(\overline{h}(y\gamma_y^{-1})\gamma_y,\overline{h}(\overline{\iota}(y,x)\gamma_y^{-1})\gamma_y\big) \nonumber \\
\text{(D is $\Gamma$-invariant)}\quad &= \epsilon_0 - D\big(\overline{h}(y\gamma_y^{-1}),\overline{h}(\overline{\iota}(y,x)\gamma_y^{-1})\big) \nonumber \\
\text{(by \eqref{eq4 lm12} and \eqref{eq5 lm12})}\quad &> \epsilon_0 -\epsilon_0/2=\epsilon_0/2.
\end{align}
By proposition \ref{properties of h'}.(4)
\begin{align*}
x\in\overline{W}^u(\gamma_x,\textbf{e}) &\Rightarrow \overline{h}'(x)\in\overline{L}^u(\gamma_x),\\
\overline{\iota}(y,x)\in\overline{W}^u(\gamma_y,\textbf{e}) &\Rightarrow \overline{h}'(\overline{\iota}(y,x))\in\overline{L}^u(\gamma_y).
\end{align*}
Thus by proposition \ref{properties of h'}.(5), \eqref{eq6 lm12} and \eqref{eq2 lm12} we have
\begin{equation*}
D\big(\overline{L}^u(\gamma_x),\overline{L}^u(\gamma_y)\big)>\epsilon_1.
\end{equation*}

Suppose $\gamma=\gamma_y\gamma_x^{-1}$. We have
\begin{align}\label{eq7 lm12}
\gamma\gamma_x=\gamma_y \not\in U_{\epsilon_1}\big(\overline{L}^u(\gamma_x)\big) &\Rightarrow \gamma\not\in U_{\epsilon_1}\big(\overline{L}^u(\gamma_x\gamma_x^{-1})\big)=U_{\epsilon_1}\big(\overline{L}^u(e)\big) \nonumber \\
&\Rightarrow \gamma^n(e)\not\in U_{\epsilon_1^n}\big(\overline{L}^u(e)\big) \supset U_{2\delta_K}\big(\overline{L}^u(e)\big).
\end{align}
On the other hand,
\begin{align*}
\overline{W}^u(\gamma\gamma_x;\textbf{e})&=\overline{W}^u(\gamma_y;\textbf{e})\\
\text{($y\in\overline{W}^u(\gamma_y;\textbf{e})$)}\quad &=\overline{W}^u(y;\textbf{e})\\
\text{(by \eqref{eq3 lm12})}\quad &\subset U_{\delta_1}\big(\overline{W}^u(x;\textbf{e})\big)\\
\text{($x\in\overline{W}^u(\gamma_x;\textbf{e})$)}\quad &=U_{\delta_1}\big(\overline{W}^u(\gamma_x;\textbf{e})\big)\\
&=U_{\delta_1}\big(\overline{W}^u(e;\textbf{e})\gamma_x\big).
\end{align*}
Now we have
\begin{align}\label{eq8 lm12}
\overline{W}^u(\gamma\gamma_x;\textbf{e})\gamma_x^{-1}\subset U_{\delta_1}\big(\overline{W}^u(e;\textbf{e})\big) &\Rightarrow \overline{W}^u(\gamma\gamma_x\gamma_x^{-1};\textbf{e})\subset U_{\delta_1}\big(\overline{W}^u(e;\textbf{e})\big)\nonumber \\
&\Rightarrow \overline{W}^u(\gamma;\textbf{e})\subset U_{\delta_1}\big(\overline{W}^u(e;\textbf{e})\big) \nonumber \\
&\Rightarrow \overline{W}^u(\gamma^2;\textbf{e})\subset U_{\delta_1}\big(\overline{W}^u(\gamma;\textbf{e})\big) \nonumber \\
\text{(by induction)}\quad &\Rightarrow \overline{W}^u(\gamma^n;\textbf{e})\subset U_{\delta_1}\big(\overline{W}^u(\gamma^{n-1};\textbf{e})\big) \nonumber \\
&\Rightarrow \overline{W}^u(\gamma^n;\textbf{e})\subset U_{\delta_1^n}\big(\overline{W}^u(e;\textbf{e})\big) \subset U_{2\delta_K}\big(\overline{L}^u(e)\big) \nonumber \\
&\Rightarrow \gamma^n\in \overline{W}^u(\gamma^n;\textbf{e})\subset U_{2\delta_K}\big(\overline{L}^u(e)\big).
\end{align}
Finally, \eqref{eq7 lm12} and \eqref{eq8 lm12} make a contradiction.
\end{proof}

Let $\tilde{u}=(u_i)\in N_{\overline{f}}$ and for each $i\in\mathbb{Z}$, $\overline{f}_{u_i,u_{i+1}}$ be the lift of $f$ by $\pi$ such that $\overline{f}(u_i)=u_{i+1}$ and define
\begin{equation*}
\overline{f}^i_{\tilde{u}}=\left\{ \begin{array}{ll} \overline{f}_{u_{i-1},u_i}\circ\ldots\circ\overline{f}_{u_0,u_1} & \text{for } i>0, \\
											(\overline{f}_{u_i,u_{i+1}})^{-1}\circ\ldots\circ(\overline{f}_{u_{-1},u_0})^{-1} & \text{for } i<0, \\
											id_N  & \text{for } i=0.
						\end{array} \right.
\end{equation*}
We define a map $\tau_{\tilde{u}}=\tau_{\tilde{u}}^f:N\rightarrow(N/\Gamma)_f$ by
\begin{equation*}
\tau_{\tilde{u}}(x)=(\pi\circ\overline{f}_{\tilde{u}}^i(x))_{i=-\infty}^{\infty}\quad (x\in N).
\end{equation*}
Since $\overline{f}(e)=e$, then $\tau_{\textbf{e}}(e)=\tau_{\tilde{e}}(e)=(\pi(e))_{i=-\infty}^{\infty}$.\\

\begin{lemma}[\cite{[Ao-Hi]} Lemma 6.6.8 (1)] \label{Ao Lemma 6.6.8 (1)}
If $x\in X$ and $\tilde{u}\in N_f$ then $\pi(\overline{W}^u(x;\tilde{u}))=W^u(\tau_{\tilde{u}}(x))$.
\end{lemma}

Let $X$ be a compact metric set and $f:X\rightarrow X$ a continuous surjection. A point $x\in X$ is said to be a \emph{nonwandering point} if for any neighborhood $U$ of $x$ there is an integer $n>0$ such that $f^n(U)\cap U\neq \emptyset$. The set $\Omega(f)$ of all nonwandering points is called the \emph{nonwandering set}. Clearly $\Omega(f)$ is closed in $X$ and invariant under $f$.

$f$ is said to be \emph{topologically transitive} (here $X$ may be not necessarily compact), if there is $x_0\in X$ such that the orbit $O^+(x_0)=\{f^i(x_0):i\in\mathbb{Z}^{\geq 0}\}$ is dense in $X$. It is easy to check that if $X$ is compact, a continuous surjection
$f:X\rightarrow X$ is topologically transitive if and only if for any $U,V$ nonempty open sets there is $n > 0$ such that $f^n(U)\cap V\neq\emptyset$.

A continuous surjection $f:X\rightarrow X$ of a metric space is \emph{topologically mixing} if for nonempty open sets $U,V$ there exists $N>0$ such that $f^n(U)\cap V\neq\emptyset$ for all $n>N$. Topological mixing implies topological transitivity.\\

For continuing, we need next theorem for which proof one can see \cite{[Ao-Hi]} Theorem 3.4.4.

\begin{theorem}[Topological decomposition theorem]
Let $f:X\rightarrow X$ be a continuous surjection of a compact metric space. If $f:X\rightarrow X$ is a $TA$-map, then the following properties hold:
\begin{enumerate}
\item[(1)](Spectral decomposition theorem due to Smale) The nonwandering set, $\Omega(f)$, contains a finite sequence $B_i\; (1\leq i\leq l)$ of $f$-invariant closed subsets such that
	\begin{enumerate}
	\item[(i)]$\Omega(f)=\bigcup_{i=1}^l B_i$ (disjoint union),
	\item[(ii)]$f_{|B_i}:B_i\rightarrow B_i$ is topologically transitive.
	\end{enumerate}
Such the subsets $B_i$ are called basic sets.
\item[(2)](Decomposition theorem due to Bowen) For $B$ a basic set there exist $a > 0$ and a finite sequence $C_i\; (0\leq i\leq a-1)$ of closed subsets such that
	\begin{enumerate}
	\item[(i)] $C_i \cap C_j=\emptyset\; (i \neq j),\; f(C_i)=C_{i+1}$ and  $f^a(C_i)=C_i$,
	\item[(ii)] $B=\bigcup_{i=0}^{a-1} C_i$,
	\item[(iii)] $f^a_{|C_i}:C_i\rightarrow C_i$ is topologically mixing,
	\end{enumerate}
Such the subsets $C_i$ are called elementary sets.
\end{enumerate}
\end{theorem}

\begin{lemma}[\cite{[Su 2]} Lemma 5.4] \label{omega n gamma}
$\Omega(f)=N/\Gamma$.
\end{lemma}

\begin{lemma}\label{lm elementary set}
$N/\Gamma$ is indeed an elementary set.
\end{lemma}
\begin{proof}
By lemma \ref{fixedpoint}, let $\overline{f}:N\rightarrow N$ be the lift of $f$ such that $\overline{f}(e)=e$. By the commuting diagram:
\begin{displaymath}
\xymatrix{N \ar[r]^{\overline{f}} \ar[d]_{\pi} & N \ar[d]^{\pi} \\ N/\Gamma \ar[r]_f & N/\Gamma}
\end{displaymath}
we have, 
\begin{equation*}
f([e])=f(\pi(e))=\pi(\overline{f}(e))=\pi(e)=[e].
\end{equation*}
Therefore, $[e]$ is a fixed point of $f$. 

By lemma \ref{omega n gamma}, $\Omega(f)=N/\Gamma$. As $N$ is connected and $\pi$ is a continuous surjection then $N/\Gamma$ is connected. In the proof of part (1) of spectral decomposition theorem, they prove that basic sets are close and open. Hence by connectedness of $\Omega(f)=N/\Gamma$, it consists of only one basic set, say $B$. On the other hand, by part (2) of spectral decomposition theorem, $N/\Gamma=B$ is the union of elementary sets. There is an elementary set, say $C$, such that $[e]\in C$. Since elementary sets are disjoint, by condition $f(C_i)=C_{i+1}$, $N/\Gamma=B$ consists of only one elementary set.
\end{proof}

\begin{lemma}[\cite{[Ao-Hi]} Remark 5.3.2 (2)] \label{Ao Remark 5.3.2 (2)}
Let $f:X\rightarrow X$ be a $TA$-map of a compact metric space and let $C$ be an elementary set of $f$. If $\tilde{u}=(u_i)\in N_f$ and $x_i\in C$ for all $i\in\mathbb{Z}$ then $W^u(\tilde{x})\cap C$ is dense in $C$.
\end{lemma}

\begin{lemma}
$\bigcup_{\gamma\in \Gamma}\overline{W}^u(\gamma;\textbf{e})$ is dense in $N$.
\end{lemma}
\begin{proof}
By lemma \ref{lm gamma in w} and lemma \ref{Ao Lemma 6.6.8 (1)} we have
\begin{equation}\label{dense rel}
\bigcup_{\gamma\in \Gamma}\overline{W}^u(\gamma;\textbf{e})=\bigcup_{\gamma\in \Gamma}(\overline{W}^u(e;\textbf{e}))\gamma=\pi^{-1}(W^u(\tau_{\textbf{e}}(e))).
\end{equation}
We have $\tau_{\textbf{e}}(e)=\big(\pi(e)\big)_{i=-\infty}^{\infty}\in (N/\Gamma)_f$. On the other hand, Since by lemma \ref{lm elementary set}, $\Omega(f)=N/\Gamma$ is an elementary set, say $C$, and for $\big(\pi(e)\big)_{i=-\infty}^{\infty}$ we have $\pi(e)\in N/\Gamma=C$ for all $i\in\mathbb{Z}$, by lemma \ref{Ao Remark 5.3.2 (2)} we have
\begin{equation*}
W^u(\tau_{\textbf{e}}(e))=W^u(\tau_{\textbf{e}}(e))\cap (N/\Gamma)=W^u(\tau_{\textbf{e}}(e))\cap C
\end{equation*}
is dense in $C=N/\Gamma$. By relation \eqref{dense rel}, we have the desired result.
\end{proof}

By lemma \ref{su1 lm12}, $\overline{h}'$ is extended to a continuous map $\tilde{h}:N\rightarrow N$. From proposition \ref{properties of h'} (1), (2) and (3), and lemma \ref{Semiconjugacy}, we have $\overline{h}=\tilde{h}$ and $\overline{h}(\gamma)=\gamma$ for all $\gamma\in\Gamma$.

\begin{prop}\label{final lemma}
For all $\gamma\in\Gamma$ and $x\in N$, $\overline{h}(x\gamma)=\overline{h}(x)\gamma$.
\end{prop}
\begin{proof}
According to lemma \ref{lm Su 2 Lemma 2.4}.(4), we have
\begin{equation}\label{eq1 final lemma}
\overline{h}(x\gamma)\in\overline{L}^s(\overline{h}(x)\gamma).
\end{equation}
Suppose that $x\in\bigcup_{\gamma\in \Gamma}\overline{W}^u(\gamma;\textbf{e})$. Then there is $\gamma_x\in\Gamma$ such that $x\in\overline{W}^u(\gamma_x;\textbf{e})$. For each $\gamma\in\Gamma$ we have
\begin{equation*}
x\gamma\in\overline{W}^u(\gamma_x;\textbf{e})\gamma=\overline{W}^u\big(\gamma_x\gamma;\textbf{e}\big).
\end{equation*}
Thus
\begin{align}\label{eq2 final lemma}
\overline{h}(x\gamma) & \in\overline{h}\bigg(\overline{W}^u\big(\gamma_x\gamma;\textbf{e}\big)\bigg) \nonumber \\
\text{(by Remark \ref{rm w l})}\quad &=\overline{L}^u\big(\overline{h}(\gamma_x\gamma)\big) \nonumber \\
&=\overline{L}^u(\gamma_x\gamma).
\end{align}
On the other hand, 
\begin{align}\label{eq3 final lemma}
\overline{h}(x)\gamma&\in \overline{h}(\overline{W}^u(\gamma_x;\textbf{e}))\gamma \nonumber \\
\text{(by Remark \ref{rm w l})}\quad &=\big(\overline{L}^u(\overline{h}(\gamma_x))\big)\gamma \nonumber \\
&=\overline{L}^u(\gamma_x)\gamma \nonumber \\
\text{(by Lemma \ref{lm gamma in l})}\quad &=\overline{L}^u(\gamma_x\gamma).
\end{align}
By \eqref{eq3 final lemma}, we have $\overline{L}^u(\gamma_x\gamma)=\overline{L}^u\big(\overline{h}(x)\gamma\big)$. Therefore, by \eqref{eq2 final lemma} we have
\begin{equation}\label{eq4 final lemma}
\overline{h}(x\gamma)\in \overline{L}^u\big(\overline{h}(x)\gamma\big).
\end{equation}
By \eqref{eq1 final lemma} and \eqref{eq4 final lemma} we have 
\begin{equation}
\overline{h}(x\gamma)\in \overline{L}^u\big(\overline{h}(x)\gamma\big) \cap \overline{L}^s\big(\overline{h}(x)\gamma\big)=\{\overline{h}(x)\gamma\}.
\end{equation}
Thus for each $x\in\bigcup_{\gamma\in \Gamma}\overline{W}^u(\gamma;\textbf{e})$ we have $\overline{h}(x\gamma)=\overline{h}(x)\gamma$. Since $\overline{h}$ is continuous and $\bigcup_{\gamma\in \Gamma}\overline{W}^u(\gamma;\textbf{e})$ is dense in $N$, we have the desired result.
\end{proof}

\textbf{The end of main theorem's proof:}
According to proposition \ref{final lemma}, $\overline{h}$ induces a homeomorphism $h:N/\Gamma\rightarrow N/\Gamma$ such that $h\circ \pi=\pi \circ \overline{h}$. i.e. the following diagram commutes:
\begin{displaymath}
\xymatrix{N \ar[r]^{\overline{h}} \ar[d]_{\pi} & N \ar[d]^{\pi} \\ N/\Gamma \ar[r]_h & N/\Gamma}
\end{displaymath}
$h$ is the conjugacy between $f$ and $A$. For if $x\in N/\Gamma$ then there is $y\in N$ such that $x=\pi(y)$ and
\begin{align*}
h\circ f(x)&=h\circ f(\pi(y))=h(f\circ\pi(y))=h(\pi\circ \overline{f}(y))\\
&=h\circ\pi(\overline{f}(y))=\pi \circ \overline{h}(\overline{f}(y))=\pi(\overline{h}\circ\overline{f}(y))\\
&=\pi(\overline{A}\circ\overline{h}(y))=\pi\circ\overline{A}(\overline{h}(y))=A\circ \pi(\overline{h}(y))\\
&=A(\pi\circ \overline{h}(y))=A(h\circ\pi(y))=A\circ h (\pi(y))\\
&=A\circ h(x).
\end{align*}
So the Main Theorem is proved.\\

\noindent{\bf Proof of Corrollary \ref{Mcor}.} As mentioned in section 2, every endomorphism of a compact metric space is a covering map. Every Anosov endomorphism is a $TA$-map (see \cite{[Ao-Hi]} Theorem 1.2.1). Every diffeomorphism is special (since it is injective). For every diffepmrphism or special expanding map of a nil-manifold, by (repaired for nil-manifolds) Theorem \ref{Su2 theorem2}, it is conjugate to a hyperbolic nil-automorphism or an expanding nil-endomorphism, respectively, which are hyperbolic nil-endomorphisms. In Theorem \ref{Theorem Main}, we prove the case that $f$ is not injective or expanding. So in this case $f$ is conjugate to a hyperbolic nil-endomorphism too.


\end{document}